\documentclass[12pt,letterpaper]{article}
\usepackage[margin=1.5in]{geometry}
\usepackage{tikz}
\usepackage[T1]{fontenc}
\usepackage{amsmath,amssymb,amstext}
\usepackage{a4wide}
\usepackage{multicol}
\usepackage{graphicx}
\usepackage{amssymb}
\usepackage{mathtools}
\usepackage{amsthm}
\usepackage{mathrsfs}
\usepackage{bbm}
\usepackage{amsfonts}
\usepackage[applemac]{inputenc}
\usepackage[english]{babel}
\usepackage{enumerate}

\usetikzlibrary{calc}
\newcommand{\Depth}{2}
\newcommand{\Height}{2}
\newcommand{\Width}{2}
\newcommand\FillCube[1][1]{
\coordinate (O) at (0,0,0);
\coordinate (A) at (0,\Width,0);
\coordinate (B) at (0,\Width,\Height);
\coordinate (C) at (0,0,\Height);
\coordinate (D) at (\Depth,0,0);
\coordinate (E) at (\Depth,\Width,0);
\coordinate (F) at (\Depth,\Width,\Height);
\coordinate (G) at (\Depth,0,\Height);
\ifx#10\relax
\else
\draw[fill=blue!40] (O) -- (C) -- (G) -- (D) -- cycle;
\fi
\draw[fill=gray!20,opacity=0.5] (O) -- (C) -- (G) -- (D) -- cycle;
\draw[fill=gray!20] (O) -- (A) -- (E) -- (D) -- cycle;
\draw[fill=blue!40] (O) -- ( $ (O)!#1!(A) $ ) -- ( $ (D)!#1!(E) $ ) -- (D) -- cycle;
\draw[fill=gray!20] (O) -- (A) -- (B) -- (C) -- cycle;
\draw[fill=blue!40] (O) -- ( $ (O)!#1!(A) $ ) -- ( $ (C)!#1!(B) $ ) -- (C) -- cycle;
\draw[fill=gray!20,opacity=0.4] (D) -- (E) -- (F) -- (G) -- cycle;
\draw[fill=blue!40,opacity=0.4] (D) -- ( $ (D)!#1!(E) $ ) -- ( $ (G)!#1!(F) $ ) -- (G) -- cycle;
\draw[fill=gray!20,opacity=0.6] (C) -- (B) -- (F) -- (G) -- cycle;
\draw[fill=blue!40,opacity=0.6] (C) -- ( $ (C)!#1!(B) $ ) -- ( $ (G)!#1!(F) $ ) -- (G) -- cycle;
\if#10\relax
\else
\draw[fill=blue!40,opacity=0.6] ( $ (O)!#1!(A) $ ) -- ( $ (C)!#1!(B) $ ) -- ( $ (G)!#1!(F) $ ) -- ( $ (D)!#1!(E) $ ) -- cycle;
\fi
}

\def\R{{ \mathbb{R}}}
\newtheorem{theorem}{Theorem}

\newtheorem{corollary}[theorem]{Corollary}

\newtheorem{definition}[theorem]{Definition}
\newtheorem{lemma}[theorem]{Lemma}
\newtheorem{proposition}[theorem]{Proposition}
\newtheorem{remark}[theorem]{Remark}
\newtheorem{hypothesis}[theorem]{Hypothesis}
\def\cal{\mathcal}

\renewcommand{\geq}{\geqslant}
\def\leq{\leqslant}

\def\cal{\mathcal}
\def\1{{ \mathbbm{1}}}

\def\di{\displaystyle}

\begin{document}
\title{Spatially Controlled Evolution of Composite Materials via Stochastic Partial Differential Equations}
\author{Nacira Agram$^{1},$ Isabelle Turpin$^{2}$,  Eya Zougar$^{2}$ }
\date{\today}
\maketitle

\footnotetext[1]{Department of Mathematics, KTH Royal Institute of Technology 100 44, Stockholm, Sweden. 
Email: nacira@kth.se. Work supported by the Swedish Research Council grant (2020-04697), the
Slovenian Research and Innovation Agency, research core funding No.P1-0448. }

\footnotetext[2]{Univ. Polytechnique Hauts-de-France, INSA Hauts-de-France, CERAMATHS - Laboratoire de Mat\'eriaux C\'eramiques et de Math\'ematiques, F-59313 Valenciennes, France. 
Emails: Isabelle.Turpin@uphf.fr, Eya.Zougar@uphf.fr}

\begin{abstract}
This paper investigates a class of controlled stochastic partial differential equations (SPDEs) arising in the modeling of composite materials with spatially varying properties. The state equation describes the evolution of a material property, influenced by control inputs that adjust the diffusivity in different spatial regions. We establish the existence of mild solutions to the SPDE under appropriate regularity conditions on the coefficients and the control. A
derivation of the sufficient and necessary conditions for optimality is provided using the stochastic maximum principle. These conditions connect the state dynamics to adjoint processes, enabling the characterization of the optimal control in terms of the curvature of the state and the sensitivity of the cost. Two explicit solvable examples are presented to illustrate the theoretical results, where the optimal control is computed explicitly for a composite material with piecewise constant diffusivity.
\end{abstract}


\section{Introduction}

The study of SPDEs, particularly those of parabolic type, has been a focal point of mathematical analysis due to their wide applicability in physics, finance and engineering. Significant progress has been made in understanding the existence and uniqueness of solutions for such equations. Notably, the concept of mild solutions, rooted in semigroup theory, has proven essential for analyzing these systems. For instance, Da Prato and Zabczyk \cite{DaPZ} laid the groundwork for mild solutions in infinite-dimensional spaces, while more recent contributions, such as those by Zili and Zougar \cite{ZZ1,ZZ5}, have refined the study of stochastic heat equation with specific operator structures, including piecewise constant coefficients.

Control problems for SPDEs have also been extensively studied. The stochastic maximum principle, as developed by Bensoussan \cite{Ben} and \O ksendal and Sulem \cite{OkSu}, provides necessary optimality conditions, while the dynamic programming principle offers a complementary approach through the derivation of the associated Hamilton-Jacobi-Bellman equations, as explored in Fleming and Soner \cite{FlSo}. Among significant contributions in this area, Debussche, Fuhrman and Tessitore \cite{opt} investigated SPDEs in the framework of Hilbert spaces, developing fixed-point techniques and variational methods to establish existence and uniqueness results for optimal solutions. Their approach focuses particularly on semi-linear SPDEs, where the control appears in the drift term or boundary conditions. Similarly, \O ksendal and Draouil \cite{OkDr} explored stochastic control problems using stochastic variational calculus and the stochastic maximum principle. Their work includes applications to systems governed by SPDEs, with a particular emphasis on backward propagation methods and backward stochastic differential equations (BSDEs). Their theoretical framework addresses problems involving controls acting either continuously or impulsively.

These contributions have opened new perspectives in the analysis of stochastic systems governed by SPDEs and have inspired numerous subsequent studies on applications in quantitative finance, risk management, and the control of smart materials. In this work, we consider a SPDE framework to model the behavior of composite materials over time. Composite materials, widely used in engineering and industrial applications, consist of multiple phases (e.g., fibers and matrix) with distinct mechanical or thermal properties. These materials are often subject to spatially heterogeneous conditions, where properties such as elasticity, conductivity, strength, or density vary across different regions. The temporal evolution of these properties, such as temperature,  is of critical importance in applications ranging from structural engineering to electronics manufacturing.

To capture these dynamics, we employ a controlled SPDE model of the form:
\begin{equation*}
\left\{
\begin{array}{rcl}
 dY(t,x)&=& \mathcal{A}_uY(t,x) \, dt + b(t,x,Y(t,x), u(t)) \, dt\\ 
 &+& \sigma(t,x,Y(t,x), u(t)) \, dB(t); \quad (t,x) \in [0,T] \times \mathbb{R}, \\
Y(0,x) &=& \xi(x), \quad \forall x \in \mathbb{R},
\end{array}
\right.
\end{equation*}
where \( Y(t, x) \) represents the state of the material at time \( t \) and spatial position \( x \), such as temperature, and \( B(t) \) is a zero-mean stochastic process with covariance:
\[
\mathbb{E}[B_t B_s] = t \land s.
\]

The operator \( \mathcal{A}_u\), which governs the material's response, is defined as:
\begin{equation*}
\mathcal{A}_u= \frac{1}{2 \rho(x)} \frac{d}{dx} \left( a(x, u) \rho(x) \frac{d}{dx} \right),
\end{equation*}
where \( a(x, u) \) and \( \rho(x) \) are piecewise functions representing the properties of the composite material. Specifically:
\[
a(x, u) = a_1(u)  \mathbbm{1}_{\{ x \leq 0 \}} + a_2(u)  \mathbbm{1}_{\{ x > 0 \}}, \quad 
\rho(x) = \rho_1  \mathbbm{1}_{\{ x \leq 0 \}} + \rho_2  \mathbbm{1}_{\{ x > 0 \}}.
\]
Here, \( a_i \in C^1(\mathbb{R}) \) for \( i = 1, 2 \), and \( \rho_1, \rho_2 \) are positive constants. The spatial variation in \( a(x, u) \) allows us to model the heterogeneous nature of composite materials, while the control input \( u(t) \), independent of \( x \), represents uniform external factors such as temperature settings or mechanical forces applied across the material.

The control function \( u(t) \) could correspond to external factors such as:
\begin{itemize}
    \item Uniform temperature settings during curing or heating processes.
    \item Global mechanical loads  applied to the material.
    \item Homogeneous processing parameters such as pressure or chemical treatments.
\end{itemize}

The objective is to optimize the performance of the materials by minimizing a cost functional that reflects deviations from desired metrics (e.g.,  temperature distribution) and the cost of applying control \( u(t) \). By appropriately adjusting \( u(t) \), this approach enables globally consistent modifications to the material's behavior, ensuring the final product meets specific performance requirements effectively.

\begin{figure}[ht!]
\begin{center}
\begin{tikzpicture}

\draw[ thick] (0,2.5) node[below] {\bf Material 2} -- ++(0,1.5);
\draw[ thick] (0,1.2) node[below] {\color{red}$\bullet$} -- ++(0,1.5);
\draw[ultra thick] (0,0.2) node[below] { \bf Material 1} -- ++(0,1.5);
\draw[ultra thick, ->] (0,-2) ++ (-50:.75) arc (-50:300:.75 and .25);
\draw[ultra thick] (0,-2.8) node[below] {Spatial axis} -- ++(0,1.5);
\draw[ultra thick,dashed] (0,-1.5) -- ++(0,4.5);
\draw (0,-1) circle (1.5 and .5);
\fill[blue,shade,semitransparent] (-1.5,-1) arc (-180:0:1.5 and .5) -- ++(0,2) arc (0:180:1.5 and .5) -- cyclenode[midway,left]{\small $a(x,u) =  a_1(u)$};
\draw (-1.5,-1) arc (-180:0:1.5 and .5) -- ++(0,2) arc (0:180:1.5 and .5) -- cycle;
\draw (-1.5,1) arc (-180:0:1.5 and .5);
\draw (1.5,2) node[right] {\color{blue} Region 2: \color{blue} $x> 0$};
\draw[blue,ultra thick,<-] (1.5,1)-- (2,1) node[right] {\color{blue} $x= 0$};

\draw[red,ultra thick,<-] (-1.5,0.2)-- (-2,0.9) node[left] {\color{red} Control $u(t)$};

\draw[red,ultra thick,<-] (-1.5,1.8)-- (-2,0.9);

\draw[gray,ultra thick,<-] (1.5,-1)-- (2,-1.6) node[right] {Desired metric  };
\draw[gray] (2,-2.)node[right] { (e.g,  Heat, etc ) };

\draw[red,shade,semitransparent] (-2,1.7) node[left] { $\rho(x)=\rho_2$};

\draw[blue,shade,semitransparent] (-2,-0.5) node[left] { $\rho(x)=\rho_1$};


\draw (0,1.1) circle (1.5 and .5);

\fill[red,shade,semitransparent] (-1.5,1.1) arc (-180:0:1.5 and .5) -- ++(0,2) arc (0:180:1.5 and .5) -- cycle node[midway,left]{\small $a(x,u) =  a_2(u)$};

\draw (-1.5,1.1) arc (-180:0:1.5 and .5) -- ++(0,2) arc (0:180:1.5 and .5) -- cycle;

\draw (-1.5,3.1) arc (-180:0:1.5 and .5);

\draw (1.5,0) node[right] {\color{blue} Region 1: \color{blue} $x\leq  0$};

\draw[ultra thick,->] (0,3) -- ++(0,1.5) node[align=center] {$x:$  Spatial Domain\\ };
\end{tikzpicture}

\end{center}

\caption{An example  of the composite material.}
\end{figure}
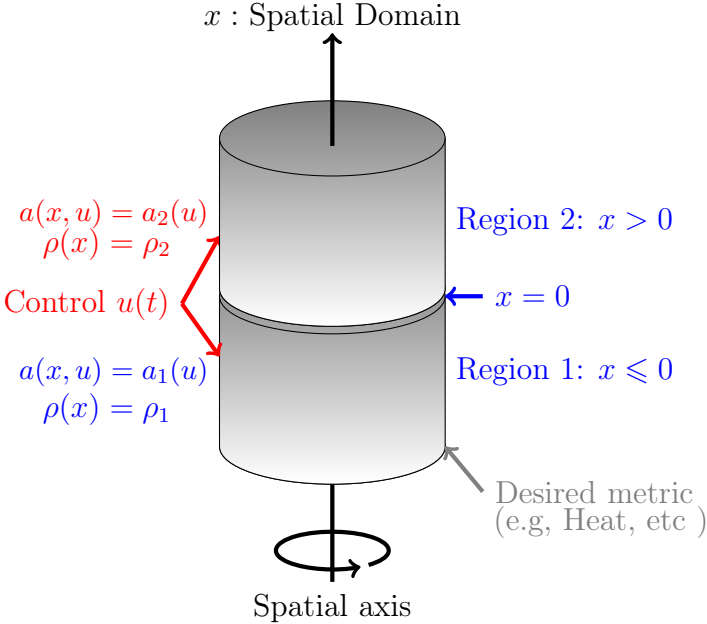

The paper is organized as follows: In Section \ref{sec2}, we focus on the  SPDE, discussing key properties of the operator and proving the existence of a mild solution. Section \ref{sec3} is devoted to deriving sufficient and necessary conditions of optimality for this class of control of SPDEs. Finally, in Section \ref{sec4}, we provide two solved examples to illustrate the theoretical results.

\section{SPDE representation }\label{sec2}
Let $T>0$, $(\Omega,\cal{F},{P})$ be a probability space with filtration $\mathbb{F}$ = $(\cal F_t)_{t \in [0,T]}$ satisfying the usual conditions and let $B$ be a one-dimensional $\mathbb{F}$-Brownian motion.

Let us consider the following general SPDE:
\begin{equation}\label{equation eq:1}
 \left\{  \begin{array}{rcl}
\displaystyle  dY(t,x) &=&{\mathcal A}Y(t,x)dt+ b(t,x,Y(t,x)) dt \\&+& \sigma(t,x,Y(t,x))dB(t), \quad (t,x) \in [0,T] \times \mathbb{R},   \\
Y(0,x)&=& \xi(x),\quad \forall x\in \mathbb{R}
\end{array} \right.
\end{equation}
where 
\begin{equation} \label{e:1.1}
{\mathcal A} = \frac{1}{2\rho (x)} \frac{d}{dx} \left( a (x)\rho(x) \frac{d}{dx} \right) ,
\end{equation}
with
\begin{equation}
\label{eq:coefA}
a(x)=  a_1 {\mathbbm 1}_{\{ x \le 0 \} } + a_2 {\mathbbm{1} }_{\{ 0 < x   \} }  \quad \hbox{and} \quad
\rho(x)= \rho_1  {\mathbbm 1}_{\{ x \le 0 \} } + \rho_2  {\mathbbm{1}}_{\{ 0 < x   \} },
\end{equation}
$a_i,\rho_i$ ($i=1, 2$) are  strictly positive constants. We note that $\di\frac{d}{d x}$ denotes the derivative in the distributional sense.
Throughout this paper, we assume the following assumptions on the coefficients $b,\sigma$ and  $\xi$.
\begin{hypothesis}\label{hypo1}
 The  coefficients $b,\sigma:\Omega\times[0,T]\times \mathbb{R} \times \R\to \mathbb{R} $ are Lipschitz continuous. That is,   there exists a constant $Lip>0,$ such that
    \begin{equation}\label{lip1}
       |b(t,x,y)-b(t,x,y')| \vee |\sigma(t,x,y)-\sigma(t,x,y')| \leq Lip\,|y-y'|, \quad \forall y,y' \in \R,
    \end{equation}
    for all $(t,x) \in [0,T] \times \mathbb{R}.$
    
    Moreover, we assume that the coefficients $b$ and $\sigma$ satisfy the the linear growth assumption. That is, there exists a constant $Lip>0$ such that
     \begin{equation}\label{lip2}
       |b(t,x,y)| \vee |\sigma(t,x,y)| \leq Lip\left( 1+|y|\right), \forall  y \in \R,
    \end{equation}
    for all $(t,x) \in [0,T] \times \mathbb{R}.$
\end{hypothesis}

\begin{hypothesis}\label{hypo2}
 The function $\xi:\R\mapsto\R$ is a non-random, measurable, and bounded function. That is, there exists a positive constant $M,$ such that $\left|\xi(x)\right|\leq M,\, \forall x\in\R.$
\end{hypothesis}

The SPDEs described in (\ref{equation eq:1}), specifically the case where $b\equiv 0$, were investigated by Zili and Zougar in \cite{ZZ1, ZZ2, ZZ3, ZZ5}, with $\sigma \equiv 1,$ as well as the case where $\sigma$ is an affine function, in \cite{ZZ4}, under different types of noise. They were considered it as a  good models for diffusion phenomena in medium consisting of two kinds of materials and
undergoing stochastic perturbations.
Another interesting reason is that the equation can be viewed as a stochastic counterpart of the deterministic heat problem, which has been studied by various authors in many areas, such
as ecology \cite{BO}, biology \cite{serge}, etc. The non-smoothness of the coefficients reflects
the heterogeneity of the medium in which the process under study propagates. One of the most important aspects is the explicit expression of its fundamental solution $G$. Further details about the  deterministic solution can
be found in Section \ref{subsec2}.

\subsection{Some properties  of the operator ${\cal A}$} 
\begin{lemma} 
 Let $Y\in L^2([0,T]\times \R ).$ Then, for any   $\displaystyle \varphi \in C^{\infty}_c([0,T]\times \R ),$ we have
\begin{equation}
\di \langle{\mathcal A}Y,\varphi\rangle
 =\di   \langle Y,{\mathcal A}^*\varphi\rangle,
\end{equation}
with respect to the measure $\rho(x)\,dx.$ Here, the adjoint operator ${\mathcal A}^*$ is given by: $${\mathcal A}^*( \varphi(s,x))= \di  \frac{1}{2\rho(x)} \left(a_2\rho_2-a_1\rho_1 \right)\frac{d \varphi(s,x)}{dx} \delta_0(x)+ {\cal A} ( \varphi(s,x)),$$ 
where, $\delta_0$ refers to the Dirac measure and $\langle{.,.\rangle}$ is the scalar production of $L^2([0,T]\times \R).$
\end{lemma}
\proof
Let $\displaystyle \varphi \in C^{\infty}_c([0,T]\times \R )$
and $Y\in L^2([0,T]\times \R),$ then  we can express the term $\langle{\mathcal A}Y,\varphi \rangle$ as well:
\begin{align*}
\langle{\mathcal A}Y,\varphi\rangle=\int_{[0,T]\times \R}{\mathcal A} Y(s, x)\,\varphi(s,x) \rho (x)\,dx\,ds 
&=  \frac{1}{2} \int_{[0,T]} I(s) ds, 
\end{align*} 
with $I$ denotes the integral 
$I(s)=\di \int_{\R}\frac{d}{dx}\left(a(x)\rho(x)\frac{dY(s,x)}{dx} \right)\varphi(s,x) \,dx,$ $\forall s\in [0,T].$

An integration by parts allows us to obtain 
\begin{equation}\label{weakI}
\begin{array}{rcl}
\di I(s) &=& \di   \left[a(x) \rho (x) \frac{d Y(s,x)}{d x} \,\varphi(s,x)\right]_{\R} -\int_{ \R}  a(x) \rho (x)  \frac{d Y(s,x)}{d x} \,\frac{d \varphi(s,x)}{d x}\,dx \\
\noalign{\vskip 2mm}
&=&-\di \int_{ \R} a(x) \rho (x)  \frac{d Y(s,x)}{d x} \,\frac{d \varphi(s,x)}{d x}\,dx,
\end{array}
\end{equation}
here, in the last equality, we are used the definition of the derivative in the sense of distributions for a locally integrable function $x\mapsto a(x) \rho (x)  \frac{d Y(s,x)}{d x} .$

Therefore, we can write
\begin{align*}
\di I(s) &=-\di \int_{ \R^*} a(x) \rho (x)  \frac{d Y(s,x)}{d x} \,\frac{d \varphi(s,x)}{d x}\,dx
\\ &= -\di \sum_{i=1}^2 \int_{ {\cal D}_i} a_i \rho_i  \frac{d Y(s,x)}{d x} \,\frac{d \varphi(s,x)}{d x}\,dx,
\end{align*}
with ${\cal D}_1=(-\infty,0)$ and ${\cal D}_2=(0,+\infty).$ Again, by an integration by parts, we get

$$\int_{{\cal D}_i} \frac{d Y(s,x)}{d x} \, \frac{d \varphi(s,x)}{d x}  \,dx = (-1)^{i+1}  Y(s,0)\frac{d \varphi(s,0)}{d x} - \int_{{\cal D}_i}  Y(s,x) \, \frac{d^2 \varphi(s,x)}{d ^2x}  \,dx,$$
for any $i=1,2.$ Thus, 
\[ I(s)= \di \left[a_2\rho_2-a_1\rho_1 \right]Y(s,0)\frac{d \varphi(s,0)}{d x} +\sum_{i=1}^2 \int_{ {\cal D}_i} a_i \rho_i  Y(s,x) \frac{d^2 \varphi(s,x)}{d ^2x}  \,dx. \]
Therefore, it implies
\[ \begin{array}{rcl} &&\di \langle{\mathcal A}Y,\varphi\rangle
\\&=& \di  \frac{1}{2} \left[a_2\rho_2-a_1\rho_1 \right] \int_{[0,T]}Y(s,0)\frac{d \varphi(s,0)}{dx}ds +\int_{[0,T]\times\R^*}  Y(s,x) {\cal A} ( \varphi(s,x)) \rho(x) \,dx\,ds 
\\&=&  \di  \frac{1}{2} \left[a(0^+)\rho(0^+)-a(0^-)\rho(0^-) \right] \int_{[0,T]}Y(s,0)\frac{d \varphi(s,0)}{dx}ds \\ && \di  +\int_{[0,T]\times\R}  Y(s,x) {\cal A} ( \varphi(s,x)) \rho(x) \,dx\,ds 
\\&=&  \di  \frac{1}{2} \left[a(0^+)\rho(0^+)-a(0^-)\rho(0^-) \right] \int_{[0,T]}Y(s,0)\frac{d \varphi(s,0)}{dx} ds \\ && \di  +\int_{[0,T]\times\R}  Y(s,x) {\cal A} ( \varphi(s,x)) \rho(x) \,dx\,ds 
\\&=& \di \langle Y, {\mathcal A}^*\varphi\rangle,
 \end{array}\]
 with ${\mathcal A}^*( \varphi(s,x))= \di  \frac{1}{2\rho(x)} \left(a_2\rho_2-a_1\rho_1 \right)\frac{d \varphi(s,x)}{dx} \delta_0(x)+ {\cal A} ( \varphi(s,x)).$
 
Then, the proof is established.
\endproof
\begin{remark}
    If we assume that $Y$ is the solution to Equation (\ref{equation eq:1}) such that $ Y(t,0)=0,$ for any $t\in[0,T]$, then,  the operator $\cal{A}_{(\rho,a,b)}$ is a self-adjoint operator in $L^2(\R),$ with respect to the measure $\rho(x)\,dx.$
    \end{remark}
 
\subsection{Existence and uniqueness of the mild solution to Equation (\ref{equation eq:1})}\label{subsec2}

\begin{definition}
We call a mild solution to SPDE  defined in $(\ref{equation eq:1})$ the stochastic process:
\begin{equation}\label{mild}
\begin{array}{rcl}
  Y(t,x)&=& {   \di  \int_\R G(t,x,z)\,\xi(z)\,dz +\int_0^t \int_{{\mathbb R}} G(t-s,x,z)\,b(s,z,Y(s,z))\,dz\,ds} \\&+&  \di \int_0^t \int_{{\mathbb R}} G(t-s,x,z)\,\sigma(s,z,Y(s,z))\,dz\,dB_s.
\end{array}
\end{equation} 
Here the stochastic integral with respect to $dB_s$ is the classical It\^o one with respect to one Brownian motion, and the integral with respect to $dz$ is  the Lebesgue-Stieltjess one. We note here that the function $G$ is the fundamental
solution of the deterministic problem $\partial_t Y=\cal{A}Y.$ 
\end{definition}

\begin{proposition}Let $\varphi\in L^2([0,T]\times \R^2).$
A  stochastic integral with respect to $dB_s$ in the form $$\di X_t(x)=\int_0^t \left(\int_{\R} \varphi(t-s,x,z)\,dz\right) dB_s$$  exists if, and only if, the integrand $$\displaystyle F^x_{t,.}= \int_\R \varphi(t-.,x,z)\,dz$$ is square-integrable with respect to Lebesgue measure, and in that case, we have
\begin{equation}
\label{Cov}
{\mathbb E}\left[X_t(x)X_s(x)\right]= \di \int_0^{t\wedge s}  \left( \int_{\R} \varphi(t-u ,x,z)\,dz
 \int_\R \varphi(s-u,x,z)\,dz \right) du
 \end{equation}
\begin{equation}
 \label{iso}
 \text{and}\quad {\mathbb E}\left[X_t^2(x)\right]=   \int_0^{t}  \left( \int_\R \varphi(t-u,x,z)\,dz\right)^2 du,
 \end{equation}
 for every $t,s \in [0,T]$ and $x\in \R.$
 \end{proposition}

An explicit expression of the fundamental solution associated to $\partial_t Y=\cal{A} Y$ is given by the following proposition, as demonstrated in \cite{CZ}.
\begin{proposition}There exists a unique fundamental solution $ G(t-s,x, z)$ of the deterministic
partial differential Equation. It can be
explicitly expressed as well:

\begin{equation} \label{e:4}
 \begin{array}{rcl}
\di  G( t-s, x,z)
 &=& \di  \frac{{ \mathbbm{1}}_{ \{ t > s \} } }{\sqrt{2 \pi (t-s)}}\left( \frac{ {\mathbbm{1}}_{ \{ z \le 0 \} } }{\sqrt{a_1}}+ \frac{ {\mathbbm{1}}_{ \{ z > 0 \} } }{\sqrt{a_2}}\right)\di\left\{ \exp \left( - \frac{(h(x) - h(z))^2}
{2 (t-s)}\right) \right. \\&+&  \displaystyle \left.
\lambda \, {\rm sign} (z)\, \exp
\left( - \frac{(|h(x)|  + |h(z)|)^2}{2(t-s)}\right) \right\},
\end{array}
\end{equation}
where
\begin{equation}\label{beta} 
\lambda= \frac{\rho_2\sqrt{a_2}- \rho_1\sqrt{a_1} }{\rho_2\sqrt{a_2}+ \rho_1\sqrt{a_1} },\quad h(z)= \frac{z}{\sqrt{a_1}}{\mathbbm 1}_{\{z\leq 0\}}+\frac{z}{\sqrt{a_2}} {\mathbbm{1}}_{\{z>0\}},
\quad  {\rm sign}(z)=
\begin{cases}
-1\hspace*{0.3cm} if \;z\leq 0\\ \;\;\;1\hspace*{0.3cm} if \;z> 0.
\end{cases}
\end{equation}
\end{proposition}

The following lemmas provide certain properties of the deterministic function $G$.
\begin{lemma}
   The fundamental solution $G$ is a positive function for any given $t,s \in (0,T] $ and $x,z\in \R.$
\end{lemma}
\begin{proof}
 By using the way that the function $x \mapsto e^{-x}$ is decreasing for any $x\in \R+,$ and Cauchy Schwartz inequality,    one can easily prove that  for any $t,s \geq 0$ and $x,y\in \R,$ 
 $$ \exp \left( - \frac{(h(x) - h(z))^2}
{2 (t-s)}\right)  \geq \exp
\left( - \frac{(|h(x)|  + |h(z)|)^2}{2(t-s)}\right). $$
Therefore, it implies
\[ G( t-s, x,z) \geq   \frac{ \gamma(x) }{\sqrt{2 \pi (t-s)}} \min_{i=1,2}\left(  \frac{ 1} {\sqrt{a_i}}\right) \exp \left( - \frac{(h(x) - h(z))^2}
{2 (t-s)}\right)  \]
with $\gamma(x)= {\mathbbm{1}}_{\lambda sign(x) \geq 0 }+\left(1+ \lambda sign(x) \right){\mathbbm{1}}_{\lambda sign(x) \leq 0 }.$ Since it is easy to prove that  $|\lambda|\leq 1,$ then, we can immediately deduce that $\gamma(x)>0,\, \forall x\in \R.$ Then, the result is achieved. 
\end{proof}

\begin{lemma}\label{maj}
    For any $0\leq s<t\leq T,$ and $x\in\R,$ there exist two positive constants $C^\lambda_{a_1,a_2}$ and $C_{a_1,a_2},$ such that
    \begin{equation}
        G(t-s,x,y)\leq \frac{C^\lambda_{a_1,a_2}}{\sqrt{2\pi(t-s)}}\,\exp \left( - \frac{(x-z)^2}
{2C_{a_1,a_2} (t-s)}\right) 
    \end{equation}
    with $C^\lambda_{a_1,a_2}=(1+|\lambda|)\di\sum_{i=1}^2 \frac{1}{\sqrt{a_i}}$ and $C_{a_1,a_2}=\di \min_{i=1,2} \frac{1}{{a_i}}.$
\end{lemma}
\begin{proof}
    The proof of this lemma can be deduced immediately from Lemma $2.1$ in \cite{ZZ4} and Lemma $3$ in \cite{ZZ3}.
\end{proof}
\begin{corollary}\label{cormaj}Fix $t \in(0,T],$ and $x\in\R,$ there exist a positive constant $c_1$ depends $\lambda,a_1,a_2,$ such that
         \begin{equation}
       \int_\R G(t-s,x,z)\,dz\leq c_1, 
    \end{equation}
    for any $s\in(0,t),$ with $c_1:=c_1(\lambda,a_1,a_2)= C^\lambda_{a_1,a_2} \left({C_{a_1,a_2}}\right)^{1/2}.$
  \end{corollary}
\begin{proof}
    Let $t,s\in [0,T],x\in \R.$ From Lemma \ref{maj} and by the change variable $y=\frac{x-z}{\sqrt{2C_{a_1,a_2} (t-s)}},$ we obtain
\[\int_\R G(t-s,x,z)\,dz\leq \frac{C^\lambda_{a_1,a_2} \sqrt{C_{a_1,a_2}}}{\sqrt{\pi}}\int_\R e^{-y^2}\,dy =  C^\lambda_{a_1,a_2} \left({C_{a_1,a_2}}\right)^{1/2}.\]
  Then, the proof is completed. 
\end{proof}
\subsubsection{Linear case}
In this context, we focus on the specific case where $b\equiv 0$ and $\sigma \equiv \sigma_0 \in \R.$

\begin{theorem}\label{existence} 
The mild solution $\left\{Y(t,x), t\in[0,T], x\in\R \right\}$ to the SPDE (\ref{equation eq:1}) is well-defined. Moreover,  it satisfies
\begin{equation}
\sup_{t\in[0,T]}\sup_{x\in\R} \mathbb{E} \big(|Y(t,x)|^2 \big)< \infty .
\end{equation}
\end{theorem}
\begin{proof}
 Consider $0 \le t \le T$ and $x\in\R$. By using the way that for any $a,b\in\R, (a+b)^2 \leq 2(a^2+b^2),$ and   isometry of Wiener given in (\ref{iso}), we obtain  
 \[\begin{array}{rcl}
 \mathbb{E}\big[|Y(t,x)|^2\big] &\leq&  2 \Bigg\{  \di  \left|\int_\R G(t,x,z)\,\xi(z)\,dz\right|^2 +   \di \sigma_0^2\,\mathbb{E}\left[ \left|\int_0^t \int_{{\mathbb R}} G(t-s,x,z)\,dz\,dB_s \right|^2\right]\Bigg\}
 \\&\leq&  2 \Bigg\{  \di  \left(\int_\R G(t,x,z)\,|\xi(z)|\,dz\right)^2 +\sigma_0^2\,  \di  \int_{0}^t \left( \int_{{\mathbb R}} G(t-s,x,z)\,dz\right)^2 \,ds \Bigg\}
 \\&\leq&  2 \Bigg\{ M^2 \di  \left(\int_\R G(t,x,z)\,dz\right)^2 +\sigma_0^2\,  \di  \int_{0}^t \left( \int_{{\mathbb R}} G(t-s,x,z)\,dz\right)^2 \,ds \Bigg\}
 \\&\leq& 2 \,c_1^2 \left[ M^2+ T\sigma_0^2\right]< \infty. 
 \end{array}\]
 The two last lines are obtained from the assumptions made on $\xi$ and Corollary \ref{cormaj}.
\end{proof}
\subsubsection{Non-linear case}
\begin{theorem}\label{exist2}
   We assume that Hypothesis \ref{hypo1} and \ref{hypo2}  are satisfied. Then, there exists a unique mild solution $Y=\left\{ Y(t,x), (t,x)\in [0,T]\times \R \right\}$ of Equation \ref{equation eq:1}.
\end{theorem}
\begin{proof}Consider $t\in(0,T],x\in\R.$ The proof will be done in two steps. First, however, we need to justify the isometry. Based on the definition of the mild solution provided in (\ref{mild}), there are three separate integrals to consider.  Since the function $G$ is deterministic, we can apply the Stochastic Fubini theorem and Wiener's isometry to deduce: 
\begin{equation}\label{integI}
\begin{array}{rcl}
     &&\di \mathbb{E}\left[ \left(\int_0^t \int_{{ \R}} G(t-s,x,z)\,\sigma(s,z,Y(s,z))\,dz\,dB_s \right)^2\right] \\&=& \di \int_0^t \int_{\R^2} G(t-s,x,z_1)\,G(t-s,x,z_2)\,g^\sigma(s,z_1,z_2)\,dz_1\,dz_2\,ds 
    \\ \text{and} &&\di \mathbb{E}\left[ \left(\int_0^t \int_{{ \R}} G(t-s,x,z)\,b(s,z,Y(s,z))\,dz\,ds \right)^2\right] \\&=& \di \int_{(0,t)^2}  \int_{\R^2} G(t-s_1,x,z_1)\,G(t-s_2,x,z_2) g^b(s_1,s_2,z_1,z_2)\,dz_1\,dz_2\,ds_1\,ds_2, 
     \end{array}
\end{equation}
  with $g^b(s_1,s_2,z_1,z_2)=\mathbb{E}\left[b(s_1,z_1,Y(s_1,z_1))b(s_2,z_2,Y(s_2,z_2))\right]$\\ and $g^\sigma(s,z_1,z_2)=\mathbb{E}\left[\sigma(s,z_1,Y(s,z_1))\sigma(s,z_2,Y(s,z_2))\right].$  \\
  
 { \it Step I: Uniqueness.} Suppose that $Y_1(t,x)$ and $Y_2(t,x)$  both satisfy the definition of mild solution provided in (\ref{mild}).  We wish to prove that $Y_1$ and $Y_2$ are
modifications of one another. To illustrate this point, we set: $D(t,x)=Y_1(t,x)-Y_2(t,x).$ So according to the definition given in (\ref{mild}), the term $D$ can be expressed as well:  
\[ \begin{array}{rcl}
D(t,x)&=& \di  \int_0^t \int_{\R} G(t-s,x,z) \left[b(s,z,Y_1(s,z))-b(s,z,Y_2(s,z)) \right]dz\,ds\\
&+& \di  \int_0^t \int_{\R} G(t-s,x,z) \left[\sigma(s,z, Y_1(s,z))-\sigma(s,z,Y_2(s,z)) \right]dz\,dB_s
\\&=& D_1(t,x)+D_2(t,x)\di . 
\end{array} 
\]
By using the way that for any $a,b\in\R, (a+b)^2 \leq 2(a^2+b^2),$ we have clearly that $$\mathbb{E}\left[\left|D(t,x)\right|^2\right]\leq 2\left( \mathbb{E}\left[\left|D_1(t,x)\right|^2\right]+\mathbb{E}\left[\left|D_2(t,x)\right|^2\right]\right).$$ So, according to (\ref{integI}), we obtain
\begin{equation} \label{d1d2} \begin{array}{rcl}
&&\mathbb{E}\left[\left|D_1(t,x)\right|^2\right]
\\  &\leq & \di \int_{(0,t)^2}  \int_{\R^2} G(t-s_1,x,z_1)\,G(t-s_2,x,z_2)\,\mathbb{E}\big(\Delta b (s_1,z_1)\Delta b (s_2,z_2)\big)dz_1\,dz_2\,ds_1\,ds_2
\\ \text{and}
&&\mathbb{E}\left[\left|D_2(t,x)\right|^2\right]
\\  &\leq &  \di \int_0^t \int_{\R}\int_{{ \R}} G(t-s,x,z_1)\,G(t-s,x,z_2) \di \mathbb{E}\big(\Delta \sigma (s,z_1)\Delta \sigma (s,z_2)\big)dz_1\,dz_2\,ds,
\end{array} 
\end{equation}
with $\Delta b (s,z)=b(s,z,Y_1(s,z))-b(s,z,Y_2(s,z))$ and $\Delta \sigma (s,z)=\sigma(s,z,Y_1(s,z))-\sigma(s,z,Y_2(s,z)),$ for any $s\in (0,t)$ and $z\in\R.$ Thanks to Cauchy-Schwarz inequality, one can get
\[ \begin{array}{rcl} \mathbb{E}\bigg(\Delta b(s_1,z_1)\Delta b(s_2,z_2)\bigg)  &\leq & \di \left\{\mathbb{E}\left[\left(\Delta b(s_1,z_1)\right)^2\right] \,\mathbb{E}\left[\left(\Delta b(s_2,z_2)\right)^2\right]\right\}^{1/2}  
\\&\leq & \di  \sup_{z\in\R}\left\{\mathbb{E}\left[\left(\Delta b(s_1,z)\right)^2\right] \right\}^{1/2} \sup_{z\in\R}\left\{\mathbb{E}\left[\left(\Delta b(s_2,z)\right)^2\right] \right\}^{1/2},
\\ \text{and} \quad   \mathbb{E}\bigg(\Delta \sigma(s,z_1)\Delta \sigma(s,z_2)\bigg) &\leq & \di \left\{\mathbb{E}\left[\left(\Delta \sigma(s,z_1)\right)^2\right] \,\mathbb{E}\left[\left(\Delta \sigma(s,z_2)\right)^2\right]\right\}^{1/2}\\  &\leq & \di  \sup_{z\in \R}\mathbb{E}\left[\left(\Delta \sigma(s,z)\right)^2\right].
\end{array}\]
 Therefore, this, together with Corollary \ref{cormaj} and H$\ddot o$lder inequality, allows us to obtain:
\begin{equation} \label{d1} 
\begin{array}{rcl}
\mathbb{E}\left[\left|D_1(t,x)\right|^2\right]
  &\leq & \di \int_{(0,t)^2} \di  \sup_{z\in\R}\left\{\mathbb{E}\left[\left(\Delta b(s_1,z)\right)^2\right] \right\}^{1/2} \sup_{z\in\R}\left\{\mathbb{E}\left[\left(\Delta b(s_2,z)\right)^2\right] \right\}^{1/2} \\&& \times \di  \left(\int_{\R} G(t-s_1,x,z)\,dz\right) \left(\int_{\R} G(t-s_2,x,z)\,dz\right) \,ds_1\,ds_2
 \\ &\leq & \di c_1^2\int_{(0,t)^2} \di  \sup_{z\in\R}\left\{\mathbb{E}\left[\left(\Delta b(s_1,z)\right)^2\right] \right\}^{1/2} \sup_{z\in\R}\left\{\mathbb{E}\left[\left(\Delta b(s_2,z)\right)^2\right] \right\}^{1/2}ds_1\,ds_2
  \\ &= & \di c_1^2\left(\int_{(0,t)} \di   \sup_{z\in\R}\left\{\mathbb{E}\left[\left(\Delta b(s,z)\right)^2\right] \right\}^{1/2}ds\right)^2
  \\ &\leq  & \di Tc_1^2 \int_{0}^t \di   \sup_{z\in\R}\mathbb{E}\left[\left(\Delta b(s,z)\right)^2\right]ds.
  \end{array}
  \end{equation}
  And \begin{equation} \label{d2} 
  \begin{array}{rcl}
\mathbb{E}\left[\left|D_2(t,x)\right|^2\right]
  &\leq & \di \int_{0}^t \di  \sup_{z\in\R}\mathbb{E}\left[\left(\Delta \sigma(s,z)\right)^2\right]   \left(\int_{\R} G(t-s,x,z)\,dz\right)^2 ds
 \\ &\leq & \di c_1^2\di \int_{0}^t \di  \sup_{z\in\R}\mathbb{E}\left[\left(\Delta \sigma(s,z)\right)^2\right] ds.
\end{array} 
\end{equation}
All this, with  the Lipschitz property given in (\ref{lip1}), allows us to get 
\begin{equation} \label{d} 
\begin{array}{rcl}
\mathbb{E}\left[\left|D(t,x)\right|^2\right]
 &\leq  & \di c_2 \int_{0}^t \di   \sup_{z\in\R}\mathbb{E}\left[\left|D (s,z)\right|^2\right]ds,
  \end{array}
  \end{equation}
with $c_2= 2^2\max(T,1)c_1^2 Lip^2.$ 

So, by denoting $H(t)=\di  \sup_{x\in  \R} \mathbb{E}\left[\left|D(t,x)\right|^2\right],$  one can get
$H(t) \leq c_2\di \int_{0}^t \di  H(s) ds. $ Then, according to the extension of Gronwall's  Lemma $15$ in \cite{dalang1999}, we can conclude that $H(t) \equiv 0$. This
completes the proof of uniqueness.
\\

 { \it Step II: Existence.} This parts follows a classical Picard iteration schema. Consider $(t,x)\in [0,T]\times \R.$ Let  $Y_n=\left\{ Y_n(t,x), \forall n\geq 0\right\}$ is the sequence defined as follows:
 \[ \begin{array}{rcl}
     Y_0(t,x)&=&  \di  \int_\R G(t,x,z)\,\xi(z)\,dz,
     \\\text{and} \quad Y_{n+1}(t,x)&=&   \di  \int_\R G(t,x,z)\,\xi(z)\,dz +\int_0^t \int_{\R} G(t-s,x,z)\,b(s,z,  Y_n(s,z))\,dz\,ds \\&+&  \di \int_0^t \int_{\R} G(t-s,x,z)\,\sigma(s,z,Y_n(s,z))\,dz\,dB_s.
 \end{array}\]
 Before establishing the convergence of this schema, we first check that $ \di \sup_{n\geq 0}\sup_{t\in[0,T]} \sup_{x\in\R} \mathbb{E}\left[\left|Y_n(t,x)\right|^2\right]$ is finite. 
Then, proceeding as the same techniques used in {\it step I}, one can easily prove that 
$$\begin{array}{rcl}
\mathbb{E}\left[\left|Y_{n+1}(t,x)\right|^2\right]
 &\leq& c\Big\{\di \left( \int_\R G(t,x,z)\,\xi(z)\,dz\right)^2 
\\ &+& \di  \int_{(0,t)^2}  \int_{\R^2} G(t-s_1,x,z_1)\,G(t-s_2,x,z_2)\,g^b_n(s_1,s_2,z_1,z_2)\,dz_1\,dz_2\,ds_1\,ds_2
\\&+& \di \int_0^t \int_{\R^2} G(t-s,x,z_1)\,G(t-s,x,z_2)\,g^\sigma_n(s,z_1,z_2)\,dz_1\,dz_2\,ds \Big\},
\end{array}$$
 with $g^b_n(s_1,s_2,z_1,z_2)=\mathbb{E}\Big(b(s_1,z_1,Y_n(s_1,z_1))b(s_2,z_2,Y_n(s_2,z_2))\Big)$\\ and $g^\sigma_n(s,z_1,z_2)=\mathbb{E}\Big(\sigma(s,z_1,Y_n(s,z_1))\sigma(s,z_2,Y_n(s,z_2))\Big).$  
 
By using the way that $\xi$ is bounded function, and from Corollary \ref{cormaj}, we can easily check that $\di \int_\R G(t,x,z)\,\xi(z)\,dz<Mc_1.$ As before, one can easily get 
\begin{equation}\label{majYn}
\begin{array}{rcl}
\mathbb{E}\left[\left|Y_{n+1}(t,x)\right|^2\right]
 &\leq& c_1^2\left\{ \di  M^2 +T \int_{0}^t \di   \sup_{z\in\R}\mathbb{E}\left[\left| b(Y_n(s,z))\right|^2\right]ds
+ \int_{0}^t \di   \sup_{z\in\R}\mathbb{E}\left[\left| \sigma(Y_n(s,z))\right|^2\right]ds\right\}
\\  &\leq& c_3\left\{ \di  1 + \int_{0}^t \di   \left[1+ \sup_{z\in\R}\mathbb{E}\left[\left| Y_n(s,z)\right|^2\right] \right]ds
\right\},
\end{array}
\end{equation}
with $c_3=\max\left(M^2,TLip^2,Lip^2\right)$ and the last line was obtained by   the second Lipschitz property given in (\ref{lip2}). 
Therefore, by denoting $  M_n(t)= \di\sup_{z\in\R}\mathbb{E}\left[\left|Y_n(t,z)\right|^2\right],$ we obtain 
 \begin{equation}
     M_{n+1}(t)\leq c_3 \Big\{ \di 1 
+ \di  \di  \int_0^t\left[1+ M_n(t) \right]ds  \Big\}.
 \end{equation}
So according to Gronwall's Lemma $15$ in \cite{dalang1999}, one can deduce that $\di \sup_{n\geq 0}\sup_{t\in[0,T]} M_n(t)< \infty.$
Now, in order to prove that the sequence $\left\{Y_n(t,x),\forall n\geq 0\right\}$ converges in $L^p,$
we denote $$D_n(t,x)=Y_{n+1}(t,x)-Y_n(t,x) \quad and \quad H_n(t)=\di  \sup_{x\in  \R} \mathbb{E}\left[\left|D_n(t,x)\right|^2\right].$$
Thus, using the same techniques used in { \it Step I}, one can write  
$$ H_n(t) \leq c_2 \int_{0}^t \di  H_{n-1}(s) ds. $$
Using again the Gronwall's Lemma (see Lemma 15 in \cite{dalang1999}), we deduce that $\sum_{n>0}D_n(t)$ converges uniformly on $[0,T],$ which allows us to deduce the existence. 
\end{proof}

\section{Optimization problem}\label{sec3}
In this section, we present the optimization framework. Specifically, we establish both necessary and sufficient conditions for optimality.

We consider the following controlled SPDE $Y_u=Y$:
\begin{equation}\label{equationccontrol e:1}
\left\{
\begin{array}{rcl}
 dY(t,x)&=& \mathcal{A}_uY(t,x) \, dt + b(t,x,Y(t,x), u(t)) \, dt\\ &+& \di \sigma(t,x,Y(t,x), u(t)) \, dB(t), \quad (t,x) \in [0,T] \times \mathbb{R}, \\
Y(0,x) &=& \xi(x), \quad \forall x \in \mathbb{R}.
\end{array}
\right.
\end{equation}
The operator \( \mathcal{A}_u\) is the self-adjoint operator and it is given by:
\begin{equation} \label{operatorU}
\mathcal{A}_u= \frac{1}{2 \rho(x)} \frac{d}{dx} \left( a(x, u) \rho(x) \frac{d}{dx} \right),
\end{equation}
where \( a(x,u) \) is a function that depends on the control \( u(t) \), specifically:
\[
a(x, u) = a_1(u)  \mathbbm{1}_{\{ x \leq 0 \}} + a_2(u)  \mathbbm{1}_{\{ x > 0 \}},
\]
where $a_i \in C^1(U), \forall i=1,2.$ 
Moreover, \( \rho(x) \) is a piecewise constant function:
\[
\rho(x) = \rho_1  \mathbbm{1}_{\{ x \leq 0 \}} + \rho_2  \mathbbm{1}_{\{ x > 0 \}}.
\]
The process \( u(t) = u(t,\omega) \) is our control process, assumed to take values in a given convex set \( U \subset \mathbb{R} \). We assume that \( u(t) \) is \( \mathbb{F} \)-predictable.  

A control process \( u(t) \) is called \emph{admissible} if the corresponding SPDE \eqref{equationccontrol e:1} has a unique solution.
Here, \( \mathcal{U} \) denotes the set of all admissible control processes.

The coefficients (drift and the volatility) \(b,\sigma:\Omega\times[0,T]\times \mathbb{R}\times \mathbb{R}\times U \to \mathbb{R}\) satisfy the following assumptions: There exists a constant $C>0,$ such that
\begin{hypothesis}\label{hypo1.2}

    \begin{equation}\label{lip1.2}
       |b(t,x,y,u)-b(t,x,y',u)| \vee |\sigma(t,x,y,u)-\sigma(t,x,y',u)| \leq C\,|y-y'|, \quad \forall y,y' \in \R,
    \end{equation}
    for all $(t,x) \in [0,T] \times \mathbb{R}$ and $u\in U$.\\

    And,
     \begin{equation}\label{lip2.1}
       |b(t,x,y,u)| \vee |\sigma(t,x,y,u)| \leq C\left( 1+|y|\right), \forall  y \in \R,
    \end{equation}
    for all $(t,x) \in [0,T] \times \mathbb{R}$ and $u\in U$.
\end{hypothesis}

 The mild solution corresponding to the SPDE defined in 
 (\ref{equationccontrol e:1}) is expressed as the following stochastic process:
\begin{equation*}
 \begin{array}{rcl}
  Y(t,x)&=& {   \di  \int_\R G(t,x,z)\,\xi(z)\,dz +\int_0^t \int_{{\mathbb R}} G(t-s,x,z)\,b(s,z,Y(s,z),u(t))\,dz\,ds} \\&+&  \di \int_0^t \int_{{\mathbb R}} G(t-s,x,z)\,\sigma(s,z,Y(s,z),u(t))\,dz\,dB_s.
\end{array}
\end{equation*} 
For each $u\in U,$  Theorem \ref{exist2} ensures the existence of a unique mild solution 
$Y$ associated with the SPDE in $(\ref{equationccontrol e:1}).$

The objective is to minimize the cost functional:
\begin{equation}\label{eq:cost}
J(u) = \mathbb{E} \left[ \int_0^T \int_{\mathbb{R}} f(t,x,Y(t,x), u(t)) \, \rho(x) dx \, dt + \int_{\mathbb{R}}g(Y(T,x)) \, \rho(x) dx \right],
\end{equation}
where \( f:\Omega\times[0,T]\times \mathbb{R}\times \mathbb{R}\times U \to \mathbb{R} \) is the running cost and \( g:\Omega\times \mathbb{R}\to \mathbb{R} \) is the terminal cost.\\
The Hamiltonian for this problem is defined as:
\begin{align*}
H(t,x,\mathcal{A}_uy,y,u,p,q)= p(t,x) [\mathcal{A}_uy+ b(t,x,y,u)] +  q(t,x)\sigma(t,x,y,u))+f(t,x,y,u),
\end{align*}
where the couple \( (p(t,x),q(t,x)) \) are the adjoint variables solution of the BSPDE

\[
dp(t,x) = - \Big[\mathcal{A}_u^*p(t,x) + \frac{\partial H}{\partial y}(t,x,Y(t,x), u(t),p(t,x),q(t,x)) \Big] dt + q(t,x) dB(t),
\]
with terminal conditions:
\[
p(T,x) = \frac{\partial g(Y(T,x))}{\partial y}.
\]

The functions $f,b,\sigma, g$ are continuously differentiable over $\mathbb{R}\times U$ with bounded derivatives.

Let \( \varphi, \psi  \in  L^2(\mathbb{R}) \), and let \( \alpha, \beta \in   \mathbb{R} \). Then:
\begin{equation*}
\mathcal{A}_u(\alpha \varphi + \beta \psi) = \frac{1}{2 \rho(x)} \frac{d}{dx} \left( a(x, u) \rho(x) \frac{d}{dx} (\alpha \varphi + \beta \psi) \right).
\end{equation*}

Using the linearity of differentiation and multiplication:
\begin{equation*}
\mathcal{A}_u(\alpha \varphi + \beta \psi) = \alpha \mathcal{A}_u(\varphi) + \beta \mathcal{A}_u(\psi).
\end{equation*}

The Frechet derivative of the reduced Hamiltonian \( h(\varphi) \in L^2(\mathbb{R}) \) is expressed as:
\begin{equation*}
\langle \nabla_\varphi h, \psi \rangle = \int_{\mathbb{R}} \mathcal{A}_u(\psi)(x)p(x) \, dx,
\end{equation*}
where \( p \in L^2(\mathbb{R}) \). This ensures that the inner product between \( \mathcal{A}_u(\psi) \) and \( p \) is well-defined.

Define the reduced Hamiltonian \( h(\varphi) \) as:
\begin{equation*}
h(\varphi) := H(t, x, \mathcal{A}_u(\varphi),\varphi, u, p, q).
\end{equation*}

The reduced Hamiltonian \( h(\varphi) \) changes linearly with perturbations \( \psi \) in \( \varphi \). The Frechet derivative is defined by:
\begin{equation*}
h(\varphi + \psi) = h(\varphi) + \langle \nabla_\varphi h, \psi \rangle + o(\|\psi\|),
\end{equation*}
where:
\begin{equation}\label{fd}
\langle \nabla_\varphi h, \psi \rangle = \mathcal{A}_u(\psi)(x)p + \psi(x)\frac{\partial H}{\partial \varphi}.
\end{equation}

In particular, if \( H \) depends on \( \varphi \) only through \( \mathcal{A}_u(\varphi) \), the derivative simplifies to:
\begin{equation*}
\langle \nabla_\varphi h, \psi \rangle = \mathcal{A}_u(\psi)(x)p.
\end{equation*}

We will use the following convexity definition for the Hamiltonian \( H \):

The Hamiltonian \( H(t, x, \mathcal{A}_u(\varphi), \varphi, u, p, q) \) is convex with respect to \( \varphi \) and \( u \) if the following inequality holds for any \( \varphi_1, \varphi_2 \) and \( u_1, u_2 \):
\begin{align}\label{cH}
&H(t, x, \mathcal{A}_{u_2}(\varphi_2), \varphi_2, u_2, p, q) - H(t, x, \mathcal{A}_{u_1}(\varphi_1), \varphi_1, u_1, p, q)\geq \nonumber \\&
\langle \nabla_\varphi H(t, x, \mathcal{A}_{u_1}(\varphi_1), \varphi_1, u_1, p, q), \varphi_2 - \varphi_1 \rangle
+ \frac{\partial H}{\partial u}(t, x, \mathcal{A}_{u_1}(\varphi_1), \varphi_1, u_1, p, q) \big( u_2 - u_1 \big),
\end{align}
where we have used the definition of the Frechet derivative given in \eqref{fd}.

\subsection{Sufficient Optimality Condition}

\begin{theorem}\label{suffop}[Sufficient maximum principle]
Let $\widehat{u}(t) \in \mathcal{U}$ be a candidate optimal control with corresponding solution \(\widehat{Y}(t,x)\) to \eqref{equationccontrol e:1}, and let \((\widehat{p}(t,x), \widehat{q}(t,x))\) satisfy the corresponding adjoint equations.

Assume the following:
\begin{enumerate}
    \item The Hamiltonian
    $ H(t,x,\mathcal{A}_u\varphi,\varphi,u,p,q)$ is convex in \((\varphi,u)\) for each \((t,x)\).

    \item The terminal cost function \(g(y)\) is convex.

    \item The control $\widehat{u}(t)$ satisfies the \emph{minimum condition}:
    \[
    \widehat{u}(t) \in \arg\min_{v \in U} \int_{\mathbb{R}}  H(t,x,\mathcal{A}_u{Y}(t,x),{Y}(t,x),v,\widehat{p}(t,x),\widehat{q}(t,x))\rho(x)dx, \quad \text{for all } t.
    \]
\end{enumerate}
Then, \(\widehat{u}(t)\) is an optimal control.
\end{theorem}

\begin{proof}
Consider the difference in the cost functional:
\[
J(u) - J(\widehat{u}) = I_1 + I_2,
\]
where 
\[
I_1 = \mathbb{E} \left[ \int_0^T \int_{\mathbb{R}} \big( f(t,x,Y(t,x),u(t)) - f(t,x,\widehat{Y}(t,x),\widehat{u}(t)) \big) \, \rho(x)dx \, dt \right],
\]
\[
I_2 = \mathbb{E} \left[ \int_{\mathbb{R}} \big( g(Y(T,x)) - g(\widehat{Y}(T,x)) \big) \rho(x)\,dx \right].
\]
Using the convexity of \(g\), we obtain:
\[
g(Y(T,x)) - g(\widehat{Y}(T,x)) \geq g'(\widehat{Y}(T,x)) \big( Y(T,x) - \widehat{Y}(T,x) \big).
\]
Thus,
\[
I_2 \geq \mathbb{E} \left[ \int_{\mathbb{R}} \widehat{p}(T,x) \big( Y(T,x) - \widehat{Y}(T,x) \big)  \rho(x)\,dx \right].
\]
We expand \(f\) using the definition of the Hamiltonian:
\begin{align*}
f(t,x,y,u)& = H(t,x,\mathcal{A}_uy,y,u,p,q) - p(t,x) \mathcal{A}_uy(t,x) - p(t,x) b(t,x,y,u)\\& - q(t,x) \sigma(t,x,y,u).
\end{align*}
Define the shorthand notations
\begin{align*}
\widetilde{H} (t,x)&= H(t,x,\mathcal{A}_uY(t,x),Y(t,x),u(t),\widehat{p}(t,x),\widehat{q}(t,x)) \\&- H(t,x,\mathcal{A}_{\widehat{u}}\widehat{Y}(t,x),\widehat{Y}(t,x),\widehat{u}(t),\widehat{p}(t,x),\widehat{q}(t,x)) \\
\widetilde{b} (t,x)&= b(t,x,\mathcal{A}_uY(t,x),Y(t,x),u(t))-b(t,x,\mathcal{A}_{\widehat{u}}\widehat{Y}(t,x),\widehat{Y}(t,x),\widehat{u}(t))
\\
\widetilde{\sigma} (t,x)&= \sigma(t,x,\mathcal{A}_uY(t,x),Y(t,x),u(t))-\sigma(t,x,\mathcal{A}_{\widehat{u}}\widehat{Y}(t,x),\widehat{Y}(t,x),\widehat{u}(t)).
\end{align*}

Then,
\begin{align}\label{I1}
I_1 = \mathbb{E}  \int_0^T \int_{\mathbb{R}} \left[\widetilde{H} - \widehat{p} (\mathcal{A}_uY - \mathcal{A}_{\widehat{u}}\widehat{Y}+\widetilde{b}  )- \widehat{q}\widetilde{\sigma}   \right](t,x)\rho(x)dx  dt.
\end{align}
Let \(\widetilde{Y} = Y - \widehat{Y}\). Applying Ito's formula to \(\widehat{p}(t,x) \widetilde{Y}(t,x)\), we have:
\begin{align}\label{PY}
&\mathbb{E} \left[ \int_{\mathbb{R}} \widehat{p}(T,x) \widetilde{Y}(T,x) \rho(x)\,dx \right]\nonumber\\&=\mathbb{E} \left[ \int_0^T \int_{\mathbb{R}} \big[ \widehat{p} (\mathcal{A}_u{Y}-\mathcal{A}_{\widehat{u}}\widehat{Y}) - \widetilde{Y} (\mathcal{A}_{\widehat{u}}^*\widehat{p}+\frac{\partial \widehat{H}}{\partial y} )+ \widehat{p}\widetilde{b} + \widehat{q}\widetilde{\sigma} \big] (t, x) \rho(x)\,dx\, dt\right].
\end{align}
Using boundary conditions and integrating by parts with \(\mathcal{A}_u\), we simplify:
\small
\begin{align}\label{IP}
 \int_{\mathbb{R}} \widetilde{Y}(t,x) \mathcal{A}_{\widehat{u}}^*\widehat{p} (t,x)\rho(x)\,dx=\int_{\mathbb{R}} \widehat{p}(t,x) \mathcal{A}_{\widehat{u}}\widetilde{Y} (t,x) \rho(x)\,dx.
\end{align}
Summing \(I_1\) \eqref{I1} and \(I_2\) \eqref{PY} , together with \eqref{IP} , we obtain:
\[
J(u) - J(\widehat{u}) \geq \mathbb{E} \int_0^T \int_{\mathbb{R}} \left[ \widetilde{H}-\widehat{p} \mathcal{A}_{\widehat{u}}\widetilde{Y} -\frac{\partial \widehat{H}}{\partial y}\widetilde{Y}  \right] (t,x) \rho(x)dxdt.
\]
Using the notion of convexity of the Hamiltonian $H$ defined in \eqref{cH}, we get
\begin{align*}
\widetilde{H}(t,x)
&\geq \langle \nabla_y H(t, x, \mathcal{A}_{\widehat{u}}\widehat{Y}(t,x), \widehat{Y}(t,x), \widehat{u}(t,x), p(t,x), q(t,x)), Y(t,x) - \widehat{Y}(t,x) \rangle\\
&+ \frac{\partial H}{\partial u}(t, x, \mathcal{A}_{\widehat{u}}\widehat{Y}(t,x), \widehat{Y}(t,x), \widehat{u}(t,x), p(t,x), q(t,x)) \big( u(t) - \widehat{u}(t) \big).
\end{align*}
Together with the minimum condition on \(H\)
\[
 \int_{\mathbb{R}} \frac{\partial H}{\partial u}(t,x,\widehat{u}(t)) (u(t) - \widehat{u}(t)) \, \rho(x)dx  \geq 0.
\]
we obtain:
\[
J(u) - J(\widehat{u}) \geq \mathbb{E} \left[ \int_0^T \int_{\mathbb{R}} \frac{\partial H}{\partial u}(t,x,\widehat{u}(t)) (u(t) - \widehat{u}(t)) \, \rho(x)dx \, dt \right] \geq 0.
\]
Thus, \(\widehat{u}(t)\) is optimal.
\end{proof}

\subsection{Necessary Optimality Condition}
This section establishes a necessary maximum principle for the given stochastic control system. Unlike sufficiency conditions, this principle does not require convexity assumptions. Instead, we impose structural conditions on the admissible control processes.

\begin{itemize}
    \item For every \( t \in [0, T] \) and any bounded \(\mathcal{F}_t\)-measurable random variable \( \theta(t, \omega) \), the control process remains within the admissible set \( \mathcal{U} \).
    
    \item Given any pair of controls \( u, \beta \in {U} \) with a uniform bound \( \|\beta(t)\| \leq K \), define the adjustment factor:
    \[
    \delta(t) = \frac{1}{2K} \text{dist}(u(t), \partial {U}) \land 1 > 0.
    \]
    and put 
    \begin{equation}\label{betafuc}
       \beta(t)=\delta(t)\beta_0(t). 
    \end{equation}   
    The modified control process is then given by
    \begin{align} \label{pu}
    \tilde{u}(t) = u(t) + a \beta(t), \quad t \in [0, T],
    \end{align}
    which remains admissible for all \( a \in (-1,1) \).
\end{itemize}

Given the perturbed control \eqref{pu}, we seek to compute the derivative
\[
\frac{d}{da} \mathcal{A}_{u + a\beta} Y^{u + a\beta}(t,x) \Big|_{a=0}.
\]

The variation of the operator \(\mathcal{A}_u\) with respect to \(u\) is denoted by  \(\frac{\partial \mathcal{A}_u}{\partial u}\) 
 
which represents the Frechet derivative  of \(\mathcal{A}_u\) with respect to \(u\).

The Gateaux derivative of \(Y^{u + a\beta}\) at \(a=0\) is denoted by
\[
\frac{d}{da} Y^{u + a\beta} \Big|_{a=0} := Z.
\]
That is, \(Z\) represents the directional derivative of \(Y^u\) along the perturbation \(\beta\).

Applying the product rule to the operator \(\mathcal{A}_{u + a\beta}\) acting on \(Y^{u + a\beta}\), we obtain:
\[
\frac{d}{da} \mathcal{A}_{u + a\beta} Y^{u + a\beta} \Big|_{a=0} =
\frac{\partial \mathcal{A}_u}{\partial u} \beta(t) Y^u +
\mathcal{A}_u \left( \frac{d}{da} Y^{u + a\beta} \Big|_{a=0} \right).
\]

The process \(Z(t, x)\) satisfies:
\[
\begin{array}{rcl}
dZ(t, x) &= & \di \bigg[\frac{\partial \mathcal{A}_u}{\partial u} \beta(t)  Y(t, x) + \mathcal{A}_u Z(t, x)+ \frac{\partial b}{\partial y}(t, x) Z(t,x)
+ \frac{\partial b}{\partial u}(t, x) \beta(t) \bigg] dt
\\&+& \di  \bigg[\frac{\partial \sigma}{\partial y}(t, x) Z(t,x)+\frac{\partial \sigma}{\partial u}(t, x) \beta(t)\bigg] dB(t),
\end{array}
\]
with the initial condition:
\begin{equation}\label{cond0}
   Z(0, x) = \frac{d}{da} Y^{u+a\beta}(0, x)\big|_{a=0} = 0. 
\end{equation}

\begin{theorem}\label{necesthm}[Necessary maximum principle] Let $\hat{u} \in \mathcal{U}$. Then the following are equivalent:
\begin{enumerate}
    \item The first-order variation satisfies:
    \[\frac{d}{da} J(\hat{u} + a \beta) \Big|_{a=0} = 0, \quad \text{for all bounded } \beta \in \mathcal{U}.
    \]
   \item The Hamiltonian condition holds:
    \[
    \frac{\partial H}{\partial u} (t,x) \Big|_{u=\hat{u}} = 0, \quad \text{for all bounded } t \in [0, T].
    \]
\end{enumerate}
\end{theorem}

\begin{proof}
For simplicity, we write $u$ instead of $\hat{u}$ in the following. The proof proceeds by decomposing the first-order variation into two terms:
\[
\frac{d}{da} J(u + a \beta) \Big|_{a=0} = I_1 + I_2,
\]
where
\[
I_1 = \frac{d}{da} \mathbb{E} \left[ \int_0^T \int_\R f(t, x, Y^{u+a\beta}(t), u(t) + a\beta(t)) \, \rho(x)dx dt \right] \Big|_{a=0}
\]
and
\[
I_2 = \frac{d}{da} \mathbb{E}\left[ \int_\R g(x, Y^{u+a\beta}(T,x)) \,\rho(x)dx\right] \Big|_{a=0}.
\]
By assumptions on $f$ and $g$, and using (\ref{cond0}), we obtain
\[
I_1 = \mathbb{E} \left[ \int_0^T \int_\R \left( \frac{\partial f}{\partial y} (t,x) Z(t,x) + \frac{\partial f}{\partial u} (t,x) \beta(t) \right) \rho(x)dxdt \right].
\]
\[
I_2 = \mathbb{E} \left[ \int_\R \frac{\partial g}{\partial y} (Y(T,x)) Z(T,x) \rho(x)dx \right] 
= \mathbb{E} \left[ \int_\R p(T,x) Z(T,x) \rho(x)dx\right].
\]
By the It\^o formula:

\begin{align*}
    I_2 &= \mathbb{E} \left[ \int_\R p(T,x) Z(T,x) \rho(x)dx \right] 
\\ & = \mathbb{E} \Big[ \int_\R \int_0^T p(t,x) dZ(t,x)dt \rho(x)dx
+ \int_\R \int_0^T Z(t,x) dp(t,x) dt \rho(x)\, dx \\ &\di + \int_\R \int_0^T d[Z, p](t,x) dt \rho(x)dx \Big]
\\ &= \mathbb{E} \Big[ \int_\R \int_0^T p(t,x) 
\left\{ \frac{\partial{ \cal A}_u}{\partial u}\beta(t)  Y(t,x) + {\cal A}_u Z(t,x) + \frac{\partial b}{\partial y} (t,x) Z(t,x) 
+ \frac{\partial b}{\partial u} (t,x) \beta(t) 
\right\} dt \rho(x)dx
\\&+  \int_\R \int_0^T p(t,x) 
\left\{ \frac{\partial \sigma}{\partial y} (t,x) Z(t,x) 
+ \frac{\partial \sigma}{\partial u} (t,x) \beta(t) 
\right\} dB(t) \rho(x)dx
\\&- \int_\R \int_0^T Z(t,x) \left[{\cal A}^*_u p(t,x) + \frac{\partial H}{\partial y} (t, x) \right] dt \rho(x)dx {  +\int_\R \int_0^T Z(t,x)q(t,x) dB(t)\rho(x)dx}
\\& {  + \int_\R \int_0^T q(t,x) 
\left\{ \frac{\partial \sigma}{\partial y} (t,x) Z(t,x) 
+ \frac{\partial \sigma}{\partial u} (t,x) \beta(t) 
\right\} dt\rho(x)dx} \Big]
\\ &
= \mathbb{E} \Big[ \int_\R \int_0^T p(t,x) \left\{ \frac{\partial{ \cal A}_u}{\partial u}\beta(t)  Y(t,x) + { \cal A}_u Z(t,x) \right\} dt \rho(x)dx
\\ & + \int_\R \int_0^T    Z(t,x) \left[  p(t,x)  \frac{\partial b}{\partial y} (t,x) + q(t,x)  \frac{\partial \sigma}{\partial y} (t,x)-  { \cal A}^*_up(t,x)-\frac{\partial H}{\partial y} (t, x) \right] dt \rho(x)dx
\\& + \int_\R \int_0^T  \beta(t)\left[ p(t,x) \frac{\partial b}{\partial u} (t,x)+q(t,x) \frac{\partial \sigma}{\partial u} (t,x)\right] \rho(x)dx dt
\Big].
\\&= -\mathbb{E} \left[ \int_\R \int_0^T Z(t,x) \frac{\partial f}{\partial y} (t, x) \rho(x)dx dt \right]
+ \mathbb{E} \left[ \int_\R \int_0^T  \left[\frac{\partial H}{\partial u} (t, x) - \frac{\partial f}{\partial u} (t,x) \right] \beta(t) \rho(x)dx dt \right]  
\\& =-I_1+\mathbb{E} \left[ \int_\R \int_0^T \frac{\partial H}{\partial u} (t, x)  \beta(t) \rho(x)dx dt \right].
\end{align*}
Therefore, we get 
\[
\frac{d}{da} J(u + a\beta) \Big|_{a=0} = I_1 + I_2 = \mathbb{E} \left[ \int_\R \int_0^T \frac{\partial H}{\partial u} (t, x) \beta(t) \rho(x)dx dt \right].
\]
We conclude that
\[
\frac{d}{da} J(u + a\beta) \Big|_{a=0} = 0
\]

if and only if

\[
\mathbb{E} \left[ \int_\R \int_0^T \frac{\partial H}{\partial u} (t, x) \beta(t) \rho(x)\,dx\, dt \right] = 0,
\]

for all bounded $\beta \in \mathcal{U}$ of the form (\ref{betafuc}).

Applying this to \( \beta(t) = \theta(t) \), 
we get that this is again equivalent to

\[
\frac{\partial H}{\partial u} (t, x) = 0, \quad \text{for all } (t, x) \in [0,T] \times \R.
\]

\end{proof}

\section{Applications}\label{sec4}
Controlled SPDEs have emerged as powerful tools for modeling and optimizing dynamic systems influenced by randomness and spatial heterogeneity. In composite materials, these equations enable precise control over key properties, such as thermal diffusivity, heat capacity, and stress distribution, under varying physical and environmental conditions. The following subsections present practical applications that demonstrate the utility of controlled SPDEs in optimizing the behavior of composite materials in different scenarios.

\subsection{Optimal Temperature Control in Composite Materials}

Optimal temperature control in composite materials is vital for maintaining desired properties and performance. During manufacturing, precise temperature regulation prevents defects like warping, cracking, or uneven hardening. Ensuring the ideal temperature range enhances strength, durability, and stability, making composites suitable for industries like aerospace, automotive, and construction. This section explores the theoretical aspects of temperature control and its impact on material quality.

So, we consider a controlled SPDE to model the behavior of a composite material. The state variable \( Y(t, x) \) represents a measurable property of the material, such as temperature,  at time \( t \) and spatial location \( x \). The system evolves under the influence of control \( u \), which adjusts the material's properties to achieve a desired behavior.

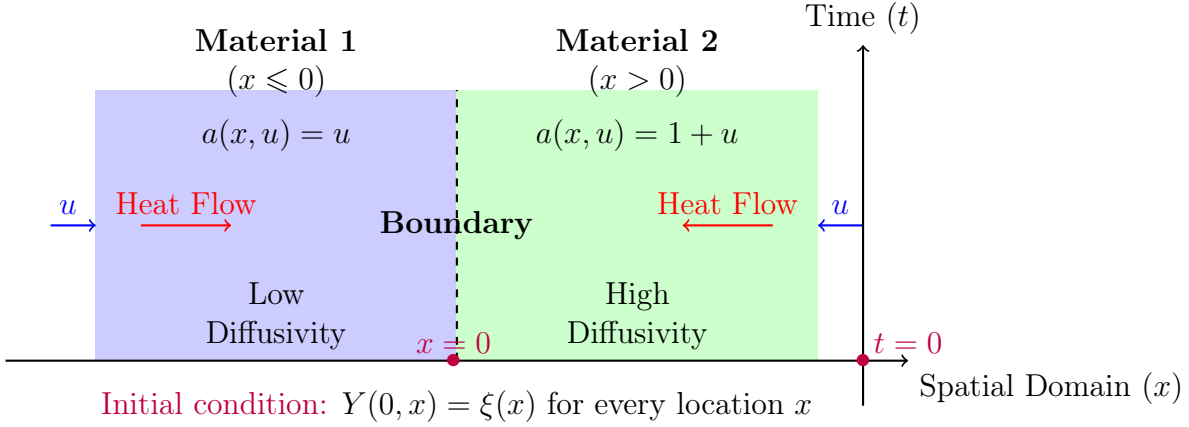
\begin{figure}[ht!]
\centering
\begin{tikzpicture}[scale=1.2, every node/.style={align=center}]

\fill[blue!20] (-4, 0) rectangle (0, 3);
\node[align=center] at (-2, 3.3) {\textbf{Material 1} \\ ($x \leq 0$)};
\node[align=center] at (-2, 2.5) {$a(x, u) = u$};

\fill[green!20] (0, 0) rectangle (4, 3);
\node[align=center] at (2, 3.3) {\textbf{Material 2} \\ ($x > 0$)};
\node[align=center] at (2, 2.5) {$a(x, u) = 1 + u$};

\draw[dashed, thick] (0, 0) -- (0, 3) node[midway] {\textbf{Boundary}};

\draw[thick, ->, red] (-3.5, 1.5) -- (-2.5, 1.5) node[midway, above] {Heat Flow};
\draw[thick, ->, red] (3.5, 1.5) -- (2.5, 1.5) node[midway, above] {Heat Flow};

\draw[thick, ->, blue] (-4.5, 1.5) -- (-4, 1.5) node[midway, above] {$u$ };
\draw[thick, ->, blue] (4.5, 1.5) -- (4, 1.5) node[midway, above] {$u$};

\node[align=center] at (-2, 0.5) {Low \\ Diffusivity};
\node[align=center] at (2, 0.5) {High \\ Diffusivity};

\node[align=center] at (0, -0.5) { {\color{purple} Initial condition:} $Y(0,x)=\xi(x)$ for every location $x$};

\draw[thick, ->] (-5, 0) -- (5, 0) node[below right] {Spatial Domain ($x$)};
\draw[thick, ->] (4.5, -0.5) -- (4.5, 3.5) node[above] {Time ($t$)};

\draw[ thick] (4.5, 0.2) node[below] {\color{purple}$\bullet$} ;
\draw[red] (4.5, 0.2) node[right] {\color{purple} $t=0$};
\draw[ thick] (-0.04, 0.2) node[below] {\color{purple}$\bullet$} ;
\draw[red] (0.5, 0.2) node[left] {\color{purple} $x=0$};

\end{tikzpicture}
\caption{Dynamics of a composite material with two regions: Material 1 ($x \leq 0$) and Material 2 ($x > 0$). The control $u$ dynamically adjusts thermal diffusivity to ensure uniform temperature distribution.}
\end{figure}

Consider a composite material with two regions:
\begin{itemize}
    \item {Material 1} ($x \leq 0$): A low-conductivity base material.
    \item {Material 2} ($x > 0$): A high-conductivity filler-enhanced region.
\end{itemize}

We assume that the initial state (at $t=0$) of the system is uniformly distributed and remains stable throughout the entire spatial domain $x$  within both materials.

The goal is to achieve a uniform temperature distribution $Y(t, x)$ across the material while minimizing the energy used to control the diffusivity $a(x, u)$.

 The evolution of temperature $Y(t, x)$ in the composite is modeled by:
\begin{equation}\label{eqex:1}
dY(t, x) = \mathcal{A}_u Y(t, x) \, dt + \sigma_0 \, dB(t),
\end{equation}
with initial condition   $Y(0,x) = \xi(x)$, for any $x\in \R,$ where the function $\xi:\R\mapsto\R$ satisfies  Hypothesis \ref{hypo2} and $\sigma_0\in\R$. And, $\mathcal{A}_u$ is the operator defined as 
\[
\mathcal{A}_u = \frac{1}{2} \frac{d}{dx} \left( a(x, u) \frac{d}{dx} \right),
\]
and the diffusivity $a(x, u)$ depends on the control $u$ as:
\[
a(x, u) =
\begin{cases}
u, & x \leq 0, \\
1 + u, & x > 0.
\end{cases}
\]

 The linear SPDE given in (\ref{eqex:1}) presents an example of the SPDE defined in (\ref{equationccontrol e:1}), that is where $b\equiv 0$ and $\sigma_0\in\R$.  Moreover, we are assumed in this case that the material densities are uniform across both components $(\rho_i=1,i=1,2).$

The control $u(t)$ increases thermal diffusivity in $x \leq 0$, compensating for the low base conductivity and adjusts the enhanced diffusivity in $x > 0$ to maintain balance.

This model describes a composite material where the spatial domain \( x \in \mathbb{R} \) is divided into two regions:

$\bullet$ Region 1 (\( x \leq 0 \)): The material has diffusivity \( a(x, u) = u \), directly controlled by \( u \).

$\bullet$ Region 2 (\( x > 0 \)): The material has diffusivity \( a(x, u) = 1 + u \), reflecting an inherent baseline diffusivity of \( 1 \), modified by \( u \).

The control \( u \) represents an external intervention, such as adjusting temperature, applying stress, or modifying conductivity, to influence the material's behavior. This setup is relevant for systems where the material's properties vary across regions, such as thermal insulation, stress distribution in structures, or spatially varying electrical conductivity.

The operator \( \mathcal{A}_u \) acts on \( Y(t, x) \) as follows:
\[
\mathcal{A}_uY(t, x) = \frac{1}{2} \frac{d}{dx} \left( a(x, u) Y' \right).
\]
For \( x \leq 0 \), where \( a(x, u) = u \), we have:
\[
\mathcal{A}_uY(t, x) = \frac{u}{2} Y''.
\]
For \( x > 0 \), where \( a(x, u) = 1 + u \), we have:
\[
\mathcal{A}_uY(t, x) = \frac{1 + u}{2}Y''.
\]

This structure is designed to capture systems where the control 
 \(u \) can change the dynamics of the state depending on the spatial location 
 \(x \) like spatially varying conductivity in a material.

Under Hypothesis 12, it follows from Theorem \ref{exist2}, that there exists a unique mild solution 
$Y.$ This uniqueness is essential for ensuring that the solution is mathematically well-posed and evolves in a consistent and stable manner.  
The existence of a mild solution plays a crucial role in describing the dynamic behavior of the system. It ensures that the state variable 
$Y(t,x)$ evolves in accordance with the combined deterministic and stochastic influences, capturing the underlying physical behavior of the composite material. This result provides a robust foundation for understanding how randomness and spatial interactions impact the material's evolution, making it a valuable tool for modeling and analysis.

The goal is to adjust \( u \) to minimize a cost functional that balances the deviation of \( Y(t, x) \) from a desired state and the effort required to control the material. The cost functional is:
\[
J(u) = \mathbb{E} \left[ \int_0^T \int_{\mathbb{R}} \left( Y(t, x)^2 + \theta u^2 \right) dx \, dt + \int_{\mathbb{R}} g(Y(T, x)) \, dx \right],
\]
where, \( Y(t, x)^2 \) penalizes the deviation of \( Y(t, x) \) from zero, \( \theta u^2 \) penalizes large control values, \( g(Y(T, x)) = \gamma Y(T, x)^2 \) is a terminal cost penalizing deviation at \( T \).

The Hamiltonian \( H \) is:
\[
H(t, x, Y, u, p, q) = p(t, x) \mathcal{A}_u Y(t, x) + \sigma_0 q(t, x) - \left( Y(t, x)^2 + \theta u^2 \right),
\]
where \( p(t, x) \) and \( q(t, x) \) are the adjoint variables. The derivative of the Hamiltonian with respect to \( u \) gives the optimality condition:
\[
\frac{\partial H}{\partial u} = p(t, x) \frac{\partial \mathcal{A}_u}{\partial u} Y(t, x) - 2 \theta u = 0.
\]
Solving for \( u \) yields:
\[
u^*(t) = \frac{p(t, x) \frac{\partial \mathcal{A}_u}{\partial u} Y(t, x)}{2 \theta} = \frac{p(t, x) \frac{1}{2} \frac{d^2 Y}{dx^2}}{2 \theta},
\]
where the derivative of \( \mathcal{A}_u \) with respect to \( u \) is computed as follows. Since \( a(x, u) = u  \mathbbm{1}_{\{ x \leq 0 \}} + (1 + u)  \mathbbm{1}_{\{ x > 0 \}} \), we have:
\[
\frac{\partial a(x, u)}{\partial u} =  \mathbbm{1}_{\{ x \leq 0 \}} +  \mathbbm{1}_{\{ x > 0 \}} = 1.
\]
Thus, the derivative of \( \mathcal{A}_u \) with respect to \( u \) is:
\[
\frac{\partial \mathcal{A}_u}{\partial u} = \frac{1}{2} \frac{d}{dx} \left( \frac{\partial a(x, u)}{\partial u} \frac{d}{dx} \right) = \frac{1}{2} \frac{d^2}{dx^2}.
\]

The optimal control adjusts the material's properties based on the curvature of \( Y(t, x) \) (captured by \( Y'' \)) and the adjoint variable \( p(t, x) \), which encodes sensitivity to the cost. High curvature or sensitivity necessitates a greater control effort to maintain the desired behavior.

This example illustrates the interplay between spatial dynamics, control, and material properties, providing insights into managing complex systems like composite materials.

\subsection{Heat Storage Control in Composite Materials}
Heat storage control in composite materials refers to managing the heat absorbed or released by the material during processes like manufacturing, curing, or thermal cycling. Proper heat management is crucial to ensure uniform temperature distribution, prevent thermal stresses, and optimize the material's performance. In composites, which often have varying thermal properties due to their multiple components (e.g., fibers and resins), controlling heat storage helps prevent issues like warping, cracking, or uneven curing. This control is particularly important in industries like aerospace or automotive, where material integrity under varying thermal conditions is essential.

Thermal energy storage systems frequently utilize composite materials with heterogeneous properties, confined within a bounded domain. These materials often consist of:
\begin{itemize}
    \item {Base material} (\(-1\leq x \leq 0\)): Low heat capacity, requiring external heating to improve energy absorption.
    \item {Enhanced material} (\( 1\geq x> 0\)): High heat capacity, where control is needed to prevent overheating or excessive heat loss.
\end{itemize}

The objective is to dynamically control the thermal diffusivity \(a(x, u)\) to achieve optimal heat storage while minimizing energy expenditure. Applications include renewable energy systems, building insulation, and phase-change materials.

The heat content \(Y(t, x)\) evolves according to the SPDE:
\begin{equation*}\label{equationccontrol2 e:1}
\left\{
\begin{array}{rcl}
 dY(t,x)&=& \mathcal{A}_uY(t,x) \, dt +  \sigma_0 \, dB(t), \quad (t,x) \in [0,T]\times [-1,1]\\
Y(0,x) &=& \xi(x), \quad \forall x \in  [-1,1],
\end{array}
\right.
\end{equation*}
where:
\[
\mathcal{A}_u = \frac{1}{2} \frac{d}{dx} \left( a(x, u) \frac{d}{dx} \right),
\]
and \(a(x, u)\) is the thermal diffusivity:
\[
a(x, u) =
\begin{cases}
u, & -1 \leq x \leq 0,  \\
1 + u, & 1 \geq x > 0.
\end{cases}
\]


 The mild solution guarantees the evolution of the heat content \( Y(t, x) \) over time, respecting the applied control \( u(t, x) \) and the inherent material properties \( a(x, u) \). .

We aim to minimize the cost functional:
\[
J(u) = \mathbb{E} \left[ \int_0^T \int_{-1}^1 \left( \gamma_1Y(t, x) + \gamma_2 u^2 \right) dx \, dt + \int_{-1}^1 \gamma_3 Y(T, x)^2 dx \right],
\]
where: 
\begin{itemize}
    \item \(\gamma_1Y(t, x)\): Penalizes deviations in heat content.
    \item \(\gamma_2 u^2\): Penalizes control effort.
    \item \(\gamma_3 Y(T, x)^2\): Ensures uniform heat distribution at \(t = T\).
\end{itemize}

The Hamiltonian for this system is:
\[
H(t, x, Y, u, p,q) = p(t, x) \mathcal{A}_u Y(t, x) +\sigma_0 q- \left( \gamma_1 Y(t, x) + \gamma_2 u^2 \right),
\]
where \((p,q)\) is the adjoint variable.

The optimal control \(u^*(t, x)\) satisfies:
\[
\frac{\partial H}{\partial u} = p(t, x) \frac{\partial \mathcal{A}_u}{\partial u} Y(t, x) - 2 \gamma_2 u = 0.
\]
Thus, the optimal control is:
\[
u^*(t, x) = \frac{p(t, x) \frac{1}{2} \frac{d^2 Y}{dx^2}}{2 \gamma_2}.
\]
The adjoint variable \(p(t, x)\) evolves backward in time according to:
\[
-\frac{\partial p(t, x)}{\partial t} = \mathcal{A}_u p(t, x) + \gamma_1+q(t,x)dB(t),
\]
with terminal condition:
\[
p(T, x) = 2 \gamma_3 Y(T, x).
\]

\textbf{Acknowledgment}\\
    The first author thanks the Laboratoire de Mat\'eriaux C\'eramiques et de Math\'ematiques in Valenciennes for their kind hospitality and support during her visit, during which this work was carried out.


\begin{thebibliography}{1}

\bibitem {Aronson} Aronson, D. G. (1968). Non-negative solutions of linear parabolic equations. Annali della Scuola Normale Superiore di Pisa-Scienze Fisiche e Matematiche, 22(4), 607-694.
\bibitem {Ben} Bensoussan, A. (1992). Stochastic Control of Partially Observable Systems. Cambridge University Press.

\bibitem{BO} Bo, L., Jiang, Y., \& Wang, Y. (2008). Stochastic Cahn-Hilliard equation with fractional noise. Stochastics and Dynamics, 8(04), 643-665.
\bibitem{CZ} Chen, Z. Q., \& Zili, M. (2015). One-dimensional heat equation with discontinuous conductance. Science China Mathematics, 58, 97-108.

\bibitem{DaPZ} Da Prato, G., \& Zabczyk, J. (2014). Stochastic Equations in Infinite Dimensions. Cambridge University Press.

\bibitem{dalang1999} Dalang, R. (1999). Extending the martingale measure stochastic integral with applications to spatially homogeneous spde's.

\bibitem{opt} Debussche, A., Fuhrman, M., \& Tessitore, G. (2007). Optimal control of a stochastic heat equation with boundary-noise and boundary-control. ESAIM: Control, Optimisation and Calculus of Variations, 13(1), 178-205.

\bibitem{FlSo} Fleming, W. H., \& Soner, H. M. (2006). Controlled Markov Processes and Viscosity Solutions. Springer.
    
\bibitem{ZZ4}Mishura, Y., Ralchenko, K., Zili, M., \& Zougar, E. (2021). Fractional stochastic heat equation with piecewise constant coefficients. Stochastics and Dynamics, 21(01), 2150002.

\bibitem{serge} Nicaise, S. (1985). Some results on spectral theory over networks, applied to nerve impulse transmission. In Polynomes Orthogonaux et Applications: Proceedings of the Laguerre Symposium held at Bar-le-Duc, October 15–18, 1984 (pp. 532-541). Springer Berlin Heidelberg.

\bibitem{OkDr} \O ksendal, B., \& Draouil, S. (2018). Optimal insider control of stochastic partial differential equations. Stochastics and Dynamics, Vol. 18, No. 01.

\bibitem{OkSu} \O ksendal, B., \& Sulem, A. (2019). Applied Stochastic Control of Jump Diffusions. Springer. 

\bibitem{ZZ1} Zili, M., \&  Zougar, E. (2019). One-dimensional stochastic heat equation with discontinuous conductance. Applicable Analysis, 98(12), 2178-2191.

\bibitem{ZZ2} Zili, M., \& Zougar, E. (2020). Exact variations for stochastic heat equations with piecewise constant coefficients and application to parameter estimation. Theory of Probability and Mathematical Statistics, 100, 77-106.

\bibitem{ZZ3} Zili, M., \& Zougar, E. (2019). Spatial quadratic variations for the solution to a stochastic partial differential equation with elliptic divergence form operator. Modern Stochastics: Theory and Applications, 6(3), 345-375.
\bibitem{ZZ5} Zili, M., \& Zougar, E. (2021). Stochastic heat equation with piecewise constant coefficients and generalized fractional type noise. Theory of Probability and Mathematical Statistics, 104, 123-144.
\end{thebibliography}
\end{document}